\documentclass{amsart}
\usepackage{graphicx}
\newtheorem{thm}{Theorem}[section]
\newtheorem{cor}[thm]{Corollary}
\newtheorem{lem}[thm]{Lemma}
\newtheorem{prop}[thm]{Proposition}
\theoremstyle{definition}

\theoremstyle{remark}

\numberwithin{equation}{section}

\begin{document}

\title[]{Integral Van Vleck's and Kannappan's  functional equations on semigroups}%
\author{ Elqorachi Elhoucien}%
\address{Ibn Zohr University, Faculty of Sciences, Department of Mathematics, Agadir, Morocco}
\email{elqorachi@hotamail.com}%

\begin{abstract}
In this paper we study  the solutions of the integral Van Vleck's
functional equation for the sine
$$\int_{S}f(x\tau(y)t)d\mu(t)-\int_{S}f(xyt)d\mu(t) =2f(x)f(y),\; x,y\in S$$ and  the integral
Kannappan's functional equation
$$\int_{S}f(xyt)d\mu(t)+\int_{S}f(x\tau(y)t)d\mu(t) =2f(x)f(y),\;
x,y\in S,$$ where $S$ is a semigroup, $\tau$ is an involution of $S$
and $\mu$  is a measure that is linear combinations of point
measures $(\delta_{z_{i}})_{i\in I}$, such that for all $i\in I$,
$z_{i}$ is  contained in the center of $S$. \\ We express the
solutions of the first equation by means of multiplicative functions
on $S$, and we prove that the solutions of the second equation are
closely related to the solutions of the classic d'Alembert's
functional equation with involution.
\end{abstract}
\maketitle
\section{Introduction}Throughout this paper $S$ denotes a semigroup, and $\tau$ :
$S\longrightarrow S$ is an involution of $S$. That is
$\tau(xy)=\tau(y)\tau(x)$ and $\tau(\tau(x))=x$ for all $x,y\in S$.
If $\mu$ denotes a discrete complex  measure, we say that $\mu$ is
$\tau$-invariant and we write $\mu=\tau(\mu)$ if
$\int_{S}f(\tau(t))d\mu(t)=\int_{S}f(t)d\mu(t)$
 for all complex-valued continous and bounded function $f$ on a topological semigroup $S$. A function $\chi$ : $S\longrightarrow \mathbb{C}$ is a
multiplicative function if $\chi(xy)=\chi(x)\chi(y)$ for all $x,y\in
S$.\\ In 2003, Elqorachi and Akkouchi  \cite{akk2} introduced and
studied the bounded and  continuous solutions
 $f\neq 0$ of the
following generalized d'Alembert's functional equation
\begin{equation}\label{eq1}
\int_{G}f(xty)d\mu(t)+\int_{G}f(xt\tau(y))d\mu(t) = 2f(x)f(y),\;x,y
\in G\end{equation} on a topological group $G$.  They proved that
under the conditions that
$\mu=\tau(\mu)$ and $f$ satisfies the Kannappan's condition\\
$\int_{G}\int_{G}f(xtysz)d\mu(t)d\mu(s)
=\int_{G}\int_{G}f(ytxsz)d\mu(t)d\mu(s)$,  $x,y,z\in G$, there
exists a generalized $\mu$-spherical function  $\psi$:
$G\longrightarrow \mathbb{C}$:
\begin{equation}\label{eqfunction}
\int_{G}\psi(xty)d\mu(t)=\psi(x)\psi(y), \;x,y\in G\end{equation}
such that $f(x)=\frac{\psi(x)+\psi(\tau(x))}{2}$ for all $x\in G$.\\
$\mu-$spherical function and related topics are studied  in
\cite{akk,akk9}
\\In the particular case when $\mu=\delta_{z_0}$ is the Dirac measure,
the functional equation (\ref{eq1}) reduces to Kannappan's
functional equation \cite{K}
\begin{equation}\label{eq2}
f(xz_0y)+f(xz_0\tau(y)) = 2f(x)f(y),\;x,y \in S.\end{equation}
Kannappan proved that any solution  $f$: $\mathbb{R}\longrightarrow
\mathbb{C}$ of (\ref{eq2}) with $\tau(y)=-y$ for all $y\in
\mathbb{R}$  is periodic, if $z_0\neq 0$. Furthermore, the periodic
solutions  has the form $f(x)=g(x-z_0)$ where $g$ is a periodic
solution of d'Alembert functional equation
\begin{equation}\label{eq3}
g(x+y)+g(x-y) = 2g(x)g(y),\;x,y \in \mathbb{R}.\end{equation}
Perkins and Sahoo \cite{P} studied the following version of
Kannappan's functional equation
\begin{equation}\label{eq4}
f(xyz_0)+f(xy^{-1}z_0) = 2f(x)f(y),\;x,y \in S\end{equation} on
groups. They found the form of any abelian solution $f$ of
(\ref{eq4}).\\
Recently, Stetk\ae r \cite{stkan} took $z_0$ in the center and
expressed the complex-valued solutions of Kannappan's functional
equation (\ref{eq2}) on semigroups in terms of solutions of
d'Alembert's functional equation
\begin{equation}\label{eq6} g(xy)+g(x\tau(y)) = 2g(x)g(y),\;x,y \in
S,\end{equation} The  complex-valued solutions of (\ref{eq6})  are
formuled by Davison \cite{davison} on monoids that need not be
commutative.\\Stetk\ae r [9, Exercise 9.18] found the complex-valued
solution of Van Vleck's functional equation
\begin{equation}\label{eq7} f(xy^{-1}z_0)-f(xyz_0) =
2f(x)f(y),\;x,y \in G,\end{equation}  when $G$ is a  not necessarily
abelian group and $z_0$ is a fixed element in the center of $G$. We refer also to \cite{V1} and \cite{V2}. \\
Perkins and Sahoo \cite{P} replaced the group inversion by an
involution $\tau$: $G\longrightarrow G$ and they obtained the
abelian, complex-valued solutions of equation
\begin{equation}\label{eq8} f(x\tau(y)z_0)-f(xyz_0) = 2f(x)f(y),\;x,y
\in G.\end{equation}  Stetk\ae r \cite{St3} extends the results of
Perkins and Sahoo \cite{P} about equation (\ref{eq7}) to the  case
where $G$ is a semigroup and the solutions are not assumed to be
abelian.\\The  main purpose of this paper is to extend  Stetk\ae r's
results  \cite{St3,stkan} to the following generalizations of Van
Vleck's functional equation for the sine\begin{equation}\label{eq09}
\int_{S}f(x\tau(y)t)d\mu(t)-\int_{S}f(xyt)d\mu(t) =2f(x)f(y), x,y\in
S,\end{equation}  and of Kannappan's functional equation
\begin{equation}\label{eq010}
\int_{S}f(xyt)d\mu(t)+\int_{S}f(x\tau(y)t)d\mu(t) =2f(x)f(y), x,y\in
S,\end{equation} where $S$ is a semigroup, $\tau$ is an involution
of $S$ and $\mu$ is a measure that is linear combinations of point
measures $(\delta_{z_{i}})_{i\in I}$, with $z_{i}$  contained in the center of $S$, for all $i\in I$.\\
We express the solutions of (\ref{eq09}) in terms of multiplicative
functions on $S$ and we prove that the solutions of (\ref{eq010})
are closely related to the solutions of the classic d'Alembert's
functional equation (\ref{eq6}).
\section{The Solutions of the integral Van Vleck's functional equation on semigroups}
In this section we obtain the complex-valued solutions of the
integral Van Vleck's functional equation   (\ref{eq09}) on
semigroups. \\The following lemmas will be used later. They  are
generalizations of   Stetk\ae r' lemmas obtained in \cite{St3} for
$\mu=\delta_{z_0}$, where $z_0$ is a fixed element in the center of
the semigroup $S$.
\begin{lem} Let $S$ be a semigroup with an involution $\tau$: $S\longrightarrow S$.
Let $\mu$ be a complex measure with support contained in the center
of $S$. Let $f$ be a non-zero solution of equation (\ref{eq09}).
Then for all $x\in S$ we have
\begin{equation}\label{eq77}
    f(x)=-f(\tau(x)),
\end{equation}
\begin{equation}\label{eq88}
    \int_{S}f(t)d\mu(t)\neq 0,
\end{equation}
\begin{equation}\label{eq99'}
    \int_{S}\int_{S}f(ts)d\mu(t)d\mu(s)=\int_{S}\int_{S}f(\tau(t)s)d\mu(t)d\mu(s)=0,
\end{equation}
\begin{equation}\label{eq99}
    \int_{S}\int_{S}f(x\tau(t)s)d\mu(t)d\mu(s)=f(x)\int_{S}f(t)d\mu(t),
\end{equation}
\begin{equation}\label{eq100}
    \int_{S}\int_{S}f(xts)d\mu(t)d\mu(s)=-f(x)\int_{S}f(t)d\mu(t),
\end{equation}
\begin{equation}\label{eq111}
   \int_{S}f(\tau(x)t)d\mu(t)=\int_{S}f(xt)d\mu(t).
\end{equation}
 \end{lem}
\begin{proof}Replacing $y$ by $\tau(y)$ in the functional equation (\ref{eq09}) and using $\tau(\tau(y))=y$ we
get
$$\int_{S}f(xyt)d\mu(t)-\int_{S}f(x\tau(y)t)d\mu(t)=2f(x)f(\tau(y))$$$$=-[\int_{S}f(x\tau(y)t)d\mu(t)-\int_{S}f(xyt)d\mu(t)]=-2f(x)f(y),$$
which implies the  formula (\ref{eq77}).\\
 By replacing
$x$ by $\tau(s)$ in (\ref{eq09}) and using (\ref{eq77}) we have
\begin{equation}\label{eq112}
\int_{S}f(\tau(s)\tau(y)t)d\mu(t)-\int_{S}f(\tau(s)yt)d\mu(t)=-2f(s)f(y)\end{equation}
for all $x,s\in S.$ By integrating the two members of equation
(\ref{eq112}) with respect to $s$ we obtain
\begin{equation}\label{eq113}
\int_{S}\int_{S}f(\tau(s)\tau(y)t)d\mu(s)d\mu(t)-\int_{S}\int_{S}f(\tau(s)yt)d\mu(s)d\mu(t)\end{equation}$$=-2f(x)\int_{S}f(s)d\mu(s).$$
By using (\ref{eq77})  we find
$$\int_{S}\int_{S}f(\tau(s)\tau(y)t)d\mu(s)d\mu(t)=-\int_{S}\int_{S}f(\tau(t)ys)d\mu(s)d\mu(t)$$
$$=-\int_{S}\int_{S}f(\tau(t)sy)d\mu(s)d\mu(t)$$
so, we obtain
$$-2\int_{S}\int_{S}f(\tau(t)sy)d\mu(t)d\mu(s)=-2f(x)\int_{S}f(s)d\mu(s),$$
which proves (\ref{eq99}). \\Setting $y=s$ in  (\ref{eq09}) and
integrating the result obtained with respect to $s$ we get by using
(\ref{eq99}) that
$$\int_{S}\int_{S}f(x\tau(s)t)d\mu(s)d\mu(t)-\int_{S}\int_{S}f(xst)d\mu(s)d\mu(t)=2f(x)\int_{S}f(s)d\mu(s)$$$$=f(x)\int_{S}f(s)d\mu(s)-\int_{S}\int_{S}f(xst)d\mu(s)d\mu(t).$$
 So, we deduce  formula  (\ref{eq100}).\\
By replacing $x$ by $xs$ in the functional equation (\ref{eq09}) and
integrating the result obtained with respect to $s$ we get by using
(\ref{eq100}) and the support of $\mu$ contained in the center of
$S$ that
$$\int_{S}\int_{S}f(xs\tau(y)t)d\mu(s)d\mu(t)-\int_{S}\int_{S}f(xsyt)d\mu(s)d\mu(t)$$$$=2f(y)\int_{S}f(xs)d\mu(s)$$
$$=-f(x\tau(y))\int_{S}f(s)d\mu(s)+f(xy)\int_{S}f(s)d\mu(s).$$ If
$\int_{S}f(s)d\mu(s)=0$, then $f(y)\int_{S}f(xs)d\mu(s)=0$ for all
$x,y\in S$. Since $f\neq 0$ then $\int_{S}f(xs)d\mu(s)=0$ for all
$x\in S$, so we have
$$\int_{S}f(x\tau(y)t)d\mu(t)-\int_{S}f(xyt)d\mu(t)=0=2f(x)f(y)$$ for
all $x,y\in S$ from which we deduce that $f(x)=0$ for all $x\in S$.
This contradicts the assumption that $f\neq 0$ and it follows that
$\int_{S}f(s)d\mu(s)\neq 0$, so, we have (\ref{eq88}).\\
From (\ref{eq100}) and (\ref{eq77}), we have
$$\int_{S}\int_{S}\int_{S}\int_{S}f(\tau(st)s't')d\mu(t)d\mu(s)d\mu(t')d\mu(s')$$
$$=-\int_{S}\int_{S}f(\tau(st))d\mu(s)d\mu(t)\int_{S}f(s)d\mu(s)=\int_{S}\int_{S}f(st)d\mu(s)d\mu(t)\int_{S}f(s)d\mu(s).$$
By setting $x=\tau(t)s'$, $y=s$ in (\ref{eq09}) and integrating the
result obtained with respect to $t$ and $s'$ we get by a computation
that
$$\int_{S}\int_{S}\int_{S}\int_{S}f(\tau(t)s'\tau(s)t')d\mu(s)d\mu(t)d\mu(s')d\mu(t')$$$$-\int_{S}\int_{S}\int_{S}\int_{S}f(\tau(t)s'st')d\mu(s)d\mu(t)d\mu(s')d\mu(t')$$
$$=2\int_{S}\int_{S}f(\tau(t)s')d\mu(t)d\mu(s')\int_{S}f(s)d\mu(s),$$ which can be
written as follows
$$\int_{S}\int_{S}f(st)d\mu(s)d\mu(t)\int_{S}f(s)d\mu(s)$$$$+\int_{S}\int_{S}f(\tau(t)s)d\mu(s)d\mu(t)\int_{S}f(s)d\mu(s)$$
$$=2\int_{S}\int_{S}f(\tau(t)s')d\mu(t)d\mu(s')\int_{S}f(s)d\mu(s).$$ This implies that
$$\int_{S}\int_{S}f(st)d\mu(s)d\mu(t)=\int_{S}\int_{S}f(\tau(t)s)d\mu(s)d\mu(t).$$
On the other hand, since $f(\sigma(x))=-f(x)$ for all $x\in S$, we
get
 $$\int_{S}\int_{S}f(\tau(t)s)d\mu(t)d\mu(s)=-\int_{S}\int_{S}f(\tau(t)s)d\mu(t)d\mu(s),$$
 and then we obtain
 $$\int_{S}\int_{S}f(\tau(t)s)d\mu(t)d\mu(s)=0=\int_{S}\int_{S}f(ts)d\mu(t)d\mu(s),$$
 which proves (\ref{eq99'}).\\
 By replacing $x$ by $st$ in (\ref{eq09})  and
integrating the result obtained with respect to $s$ and $t$ we get
$$\int_{S}\int_{S}\int_{S}f(st\tau(y)t')d\mu(s)d\mu(t)d\mu(t')$$$$-\int_{S}\int_{S}\int_{S}f(styt')d\mu(s)d\mu(t)d\mu(t')$$
$$=2f(y)\int_{S}\int_{S}f(st)d\mu(t)d\mu(s)=0.$$
From (2.5) we have
$$\int_{S}\int_{S}\int_{S}f(st\tau(y)t')d\mu(s)d\mu(t)d\mu(t')$$$$=\int_{S}[\int_{S}\int_{S}f(\tau(y)t'ts)d\mu(s)d\mu(t)]d\mu(t')$$
$$=-\int_{S}f(s)d\mu(s)\int_{S}\int_{S}f(\tau(y)t')d\mu(t')$$ and
$$\int_{S}\int_{S}\int_{S}f(styt')d\mu(s)d\mu(t)d\mu(t')d\mu(s)$$
$$=\int_{S}[\int_{S}\int_{S}f(yt'st)d\mu(s)d\mu(t)]d\mu(t')$$$$=-\int_{S}f(s)d\mu(s)\int_{S}f(yt')d\mu(t').$$
Since $\int_{S}f(s)d\mu(s)\neq0$ we deduce that
$\int_{S}f(yt')d\mu(t')=\int_{S}f(\tau(y)t')d\mu(t')$ for all $y\in
S.$ This completes  the proof.
\end{proof}
\begin{lem}Let $f$ be a non-zero solution of equation (\ref{eq09}). Then\\
(1) the function defined by
$$g(x)\;:=\frac{\int_{S}f(xt)d\mu(t)}{\int_{S}f(s)d\mu(s)}\;for \;x\in S$$ is a non-zero
abelian solution of d'Alembert's functional equation
(\ref{eq6}).\\(2) $$\int_{S}\int_{S}g(ts)d\mu(t)d\mu(s)\neq
0;\;\;\int_{S}g(s)d\mu(s)=0$$(3) The function $g$ from (1) has the
form $g=\frac{\chi+\chi\circ\tau}{2}$, where $\chi$ :
$S\longrightarrow \mathbb{C}$, $\chi\neq 0$, is a multiplicative
function on $S.$
\end{lem}
\begin{proof} (1). From (\ref{eq99}), (\ref{eq100}) and the definition of $g$ we
get by a computation that
$$(\int_{S}f(s)d\mu(s))^{2}[g(xy)+g(x\tau(y))]=$$$$\int_{S}f(s)d\mu(s)\int_{S}f(xyt)d\mu(t)+\int_{S}f(s)d\mu(s)\int_{S}f(x\tau(y)t)d\mu(t)$$
$$=-\int_{S}\int_{S}\int_{S}f(xytss')d\mu(t)d\mu(s)d\mu(s')$$
$$+\int_{S}\int_{S}\int_{S}f(x\tau(y)t\tau(s)s')d\mu(t)d\mu(s)d\mu(s')$$
$$=\int_{S}\int_{S}\int_{S}f(xs'\tau(ys)t)d\mu(t)d\mu(s)d\mu(s')-\int_{S}\int_{S}\int_{S}f(xs'yst)d\mu(t)d\mu(s)d\mu(s')$$
$$=2\int_{S}f(xs')d\mu(s')\int_{S}f(ys)d\mu(s)$$
 which implies the desired result.\\
 (2). From (\ref{eq100}) and the definition of $g$ we get  $$\int_{S}\int_{S}g(ts)d\mu(t)d\mu(s)=\frac{\int_{S}\int_{S}\int_{S}f(s'ts)d\mu(t)d\mu(s)d\mu(s')}{\int_{S}f(s)d\mu(s)}$$
 $$=\frac{-\int_{S}f(s')d\mu(s')\int_{S}f(s)d\mu(s)}{\int_{S}f(s)d\mu(s)}=-\int_{S}f(s)d\mu(s)\neq 0.$$
 From (\ref{eq99'}) and the definition of $g$ we get $$\int_{S}g(s)d\mu(s)=\frac{\int_{S}f(st)d\mu(s)d\mu(t)
 }{\int_{S}f(s)d\mu(s)}=\frac{0
 }{\int_{S}f(s)d\mu(s)}=0$$
Furthermore, $\int_{S}\int_{S}g(st)d\mu(t)d\mu(s)\neq 0$, so $g\neq
0.$\\As $g$ is a solution of equation (\ref{eq6}) then by [9,
Proposition 9.17(c)] $g$ is a solution of pre-d'Alembert function.
Now, according to [9, Proposition 8.14(a)] we discuss two cases:
\\\textbf{Case 1.} If there is a $t$ in the center of $S$ such that
$g(t)^{2}\neq d(t)$, then $g$ is  abelian.
\\\textbf{Case 2} If for all $t$ in the center of $S$ satisfies
$g(t)^{2}= d(t)$, then we get $g(xt)=g(x)g(t)$ for all $x\in S$ and
$t$ in the center of $S$.  By integrating the expression with
respect to $t$ we get
$\int_{S}g(xt)d\mu(t)=g(x)\int_{S}g(t)d\mu(t)=0$ for all $x\in S$
and then $\int_{S}\int_{S}g(st)d\mu(t)d\mu(s)=0$, This contradicts
the first assertion of Lemma 2.2 (2). Finally, we conclude that $g$
is abelian and for the rest of the proof we use [9, Theorem 9.12].
\end{proof}
The main content of this section is the following theorem.
\begin{thm} The non-zero solutions $f$ : $S\longrightarrow \mathbb{C}$ of
the functional equation (\ref{eq09}) are the functions of the form
\begin{equation}\label{eq300}
    f=[\frac{\chi- \chi\circ\tau}{2}]\int_{S}\chi(\tau(t))d\mu(t),
\end{equation}where $\chi$ : $S\longrightarrow \mathbb{C}$ is a
multiplicative function such that $\int_{S}\chi(t)d\mu(t)\neq 0$ and
$\int_{S}\chi(\tau(t)d\mu(t))=-\int_{S}\chi(t)d\mu(t)$. \\If $S$ is
a topological semigroup and that $\tau$ : $S\longrightarrow S$,  is
continuous, then the non-zero solution $f$ of equation (\ref{eq09})
is continuous, if and only if $\chi$ is continuous.
\end{thm}\begin{proof}Let $f$: $S\longrightarrow \mathbb{C}$, $f\neq 0$ , be a solution of equation (\ref{eq09}). Then by replacing $y$ by $s$ and integrating the result
obtained with respect to $s$
 we get
$$f(x)=\frac{\int_{S}\int_{S}f(x\tau(s)t)d\mu(s)d\mu(t)-\int_{S}\int_{S}f(xst)d\mu(s)d\mu(t)}{2\int_{S}f(s)d\mu(s)}$$
$$=\frac{\int_{S}g(x\tau(s))d\mu(s)-\int_{S}g(xs)d\mu(s)}{2}$$ for all $x\in S$
and where $g$ is the function given by Lemma 2.2(1). According to
$g=\frac{\chi+ \chi\circ\tau}{2}$ we get the following formula
\begin{equation}\label{eq301}
    f=[\frac{\int_{S}\chi(s)d\mu(s)-\int_{S}\chi(\tau(s))d\mu(s)}{2}][\frac{\chi\circ\tau-\chi}{2}].
\end{equation} In view of Lemma 2.1 we have
$\int_{S}f(\tau(x)t)d\mu(t)=\int_{S}f(xt)d\mu(t)$ for all $x\in S$.
Substituting (\ref{eq301})  into (\ref{eq111}) we find after simple
computations  that
$$[\int_{S}\chi(\tau(s))d\mu(s)+\int_{S}\chi(s)d\mu(s)][\chi-\chi\circ\tau]=0.$$ The rest of
the proof is similar to Stetk\ae r's proof \cite{St3}.\end{proof}
\begin{cor} \cite{St3} Let $S$ be a semigroup with an involution
$\tau$: $S\longrightarrow S$. If $\mu=\delta_{z_0}$, where $z_0$ is
a fixed element in the center of $S$. The non-zero solutions $f$ :
$S\longrightarrow \mathbb{C}$ of the functional equation (\ref{eq8})
are the functions of the form
\begin{equation}\label{eq300}
    f=\chi(\tau(z_0))[\frac{\chi- \chi\circ\tau}{2}],
\end{equation}where $\chi$ : $S\longrightarrow \mathbb{C}$ is a
multiplicative function such that $\chi(z_0)\neq 0$ and
$\chi(\tau(z_0)=-\chi(z_0)$.\end{cor}
\section{Integral Kannappan's functional equation on semigroups}
In this section we study the complex-valued solutions of the
functional equation (\ref{eq010}). The support of the discrete
complex measure $\mu$ is assumed to be contained in the center of
the semigroup $S$. \\The following useful lemma will be used later.
It's a natural generalization of Lemma 1 and Lemma 2 obtained by
Stetk\ae r \cite{stkan} for $\mu=\delta_{z_0}$.
\begin{lem} \textbf{(1)} If $f$ $S\longrightarrow \mathbb{C}$ is a solution of (\ref{eq010}), then  for
all $x\in S$ we have
\begin{equation}\label{eqo77}
    f(x)=f(\tau(x)),
\end{equation}
\begin{equation}\label{eqo88}
    \int_{S}f(t)d\mu(t)\neq 0\Longleftrightarrow f\neq 0.
\end{equation}
\begin{equation}\label{eqo999}
    \int_{S}\int_{S}f(x\sigma(t)s)d\mu(t)d\mu(s)=f(x)\int_{S}f(t)d\mu(t),
\end{equation}
\begin{equation}\label{eqo1000}
    \int_{S}\int_{S}f(xts)d\mu(t)d\mu(s)=f(x)\int_{S}f(t)d\mu(t),
\end{equation}
 \end{lem}\textbf{(2)} If $g$ $S\longrightarrow \mathbb{C}$ is a solution of d'Alembert's functional equation  (\ref{eq6}), then\\
 (i) $g(x)=g(\tau(x))$ for all $x\in S.$\\
 (ii) The following properties are equivalent
\begin{equation}\label{pr1}
 \int_{S}g(xt)d\mu(t)=\int_{S}g(x\tau(t))d\mu(t)\; \text{for all}\;x\in S\end{equation}
 \begin{equation}\label{pr2}
  \int_{S}g(xt)d\mu(t)=g(x)\int_{S}g(t)d\mu(t)\; \text{for all}\;x\in S\end{equation}
\begin{equation}\label{pr3}
\int_{S}\int_{S}g(ts)d\mu(t)d\mu(s)=(\int_{S}g(t)d\mu(t))^{2}\end{equation}
\begin{proof}  (1). The formula (\ref{eqo77}) is proved like the corresponding statement in Lemma 2.1.\\
 By putting $x=\tau(s)$ in (\ref{eq010}) and integrating the result obtained with respect  to $s$ to get
 $$\int_{S}\int_{S}f(\tau(s)yt)d\mu(t)d\mu(s)$$$$+\int_{S}\int_{S}f(\tau(s)\tau(y)t)d\mu(t)d\mu(s) =2f(y)\int_{S} f(\tau(t))d\mu(t)=2f(y)\int_{S}
 f(t)d\mu(t),$$ where the last equality holds, because $f$ satisfies (\ref{eqo77}).
 \\In view of (\ref{eqo77}), we have
 $$\int_{S}\int_{S}f(\tau(s)\tau(y)t)d\mu(t)=\int_{S}\int_{S}f(\tau(t)ys)d\mu(t)d\mu(s).$$
 So, we obtain $2\int_{S}\int_{S}f(\tau(t)ys)d\mu(t)d\mu(s)=2f(y)\int_{S}
 f(t)d\mu(t)$, which proves (\ref{eqo999}).\\By setting $y=s$ in
 (\ref{eq010}) and integrating the result obtained with respect to $s$ we get
 $$\int_{S}\int_{S}f(xst)d\mu(t)d\mu(s)$$
 $$+\int_{S}\int_{S}f(x\tau(s)t)d\mu(t)d\mu(s) =2f(x)\int_{S} f(s)d\mu(s)$$
 $$=\int_{S}\int_{S}f(xst)d\mu(t)d\mu(s)$$
 $$+f(x)\int_{S} f(s)d\mu(s),$$ which implies the formula
 (\ref{eqo1000}).\\Assume that $f$ is a solution of equation
 (\ref{eq010}) and that $\int_{S}f(t)d\mu(t)=0$. Replacing $x$ by
 $xs$, $y$ by $yt$ in  (\ref{eq010}) and integrating the result obtaind with respect
 to $s$ and $t$ we find
$$\int_{S}\int_{S}\int_{S}f(xsytk)d\mu(t)d\mu(s)d\mu(k)+\int_{S}\int_{S}\int_{S}f(xs\tau(t)\tau(y)k)d\mu(t)d\mu(s)d\mu(k)$$
$$=2\int_{S}f(xs)d\mu(s)\int_{S}f(yt)d\mu(s).$$
Since, from (3.3) we have
$$\int_{S}\int_{S}\int_{S}f(xs\tau(t)\tau(y)k)d\mu(t)d\mu(s)d\mu(k)=\int_{S}[\int_{S}\int_{S}f(xs\tau(y)\tau(t)k)d\mu(t)d\mu(k)]d\mu(s)$$
$$=\int_{S}[\int_{S}f(t)d\mu(t)f(xs\tau(y))]d\mu(s)=\int_{S}0d\mu(s)=0.$$
In view of (3.4) we have
$$\int_{S}\int_{S}\int_{S}f(xsytk)d\mu(t)d\mu(s)d\mu(k)=\int_{S}[\int_{S}\int_{S}f(xystk)d\mu(t)d\mu(k)]d\mu(s)$$
$$=\int_{S}[\int_{S}f(t)d\mu(t)f(xys)]d\mu(s)=\int_{S}0d\mu(s)=0,$$ and it
follows that $\int_{S}f(xs)d\mu(s)\int_{S}f(yt)d\mu(s)=0$ for all
$x,y\in S.$ So, we obtain
$$\int_{S}f(xyt)d\mu(t)+\int_{S}f(x\tau(y)t)d\mu(t)=2f(x)f(y)=0$$
for all $x,y\in S$. Consequently, $f(x)=0$ for all $x\in S$ and this proves (3.2).\\
(2) Let $g$ be a solution of (\ref{eq6}). Assume that
$\int_{S}g(xt)d\mu(t)=\int_{S}g(x\tau(t))d\mu(t)$ holds for all
$x\in S$. Since $g(xt)+g(x\tau(t))=2g(x)g(t)$ for all $x,t\in S$,
then by integrating the statement with respect to $t$, we get
$$\int_{S}g(xt)d\mu(t)+\int_{S}g(x\tau(t))d\mu(t)=2g(x)\int_{S}g(t)d\mu(t)=2\int_{S}g(xt)d\mu(t).$$
Conversely,
$$2\int_{S}g(xt)d\mu(t)=2g(x)\int_{S}g(t)d\mu(t)=\int_{S}[g(xt)+g(x\tau(t))]d\mu(t)$$$$=\int_{S}g(xt)d\mu(t)+\int_{S}g(x\tau(t))d\mu(t),$$
which implies that $\int_{S}g(xt)d\mu(t)=\int_{S}g(x\tau(t))d\mu(t)$
for all $x\in S$  and that (\ref{pr1}) and (\ref{pr2}) are
equivalent. \\Now,  we will show that (\ref{pr3}) and (\ref{pr2})
are equivalent. If $\int_{S}g(st)d\mu(t)=g(s)\int_{S}g(t)d\mu(t)$
for all $s\in S$, then by integration this expression with respect
to $s$, we get
$\int_{S}\int_{S}g(st)d\mu(s)d\mu(t)=(\int_{S}g(t)d\mu(t))^{2}$.
Conversely, suppose that
$\int_{S}\int_{S}g(st)d\mu(s)d\mu(t)=(\int_{S}g(t)d\mu(t))^{2}$.
Since $g$ is a solution of d'Alembert's functional equation
(\ref{eq6}), then $g$ is a solution of the pre-d'Alembert functional
equation [9, Proposition 9.17]. So, from [9, Proposition 8.14(a)] we
will discuss the following two cases.\\
\textbf{Case 1.} If for all $s$ in the center of $S$ satisfies
$g(s)^{2}=d(s)$, then $g(xs)=g(x)g(s)$ for all $x\in S$. So, by
integrating this expression with respect to $s$ we get
$\int_{S}g(xs)d\mu(s)=g(x)\int_{S}g(s)d\mu(s)$ for all $x\in S$.\\
\textbf{Case 2.} If there is $s$ in the center of $S$ such that
$g(s)^{2}\neq d(s)$, then $g$ is abelian and there exists a
multiplicative function $\chi$: $S\longrightarrow \mathbb{C}$ such
that $g=\frac{\chi+\chi\circ \tau}{2}$. Substituting this into
$\int_{S}\int_{S}g(st)d\mu(s)d\mu(t)=(\int_{S}g(t)d\mu(t))^{2}$,
gives after an elementary computations that
$$\int_{S}\chi(t)d\mu(t)-\int_{S}\chi(\tau(t))d\mu(t)=0.$$ Thus, we get $$\int_{S}g(xt)d\mu(t)=\int_{S}\frac{\chi+\chi\circ
\tau}{2}(xt)d\mu(t)$$
$$=\frac{1}{2}(\chi(x)\int_{S}\chi(t)d\mu(t)+\chi(\tau(x))\int_{S}\chi(\tau(t))d\mu(t)=\int_{S}\chi(t)d\mu(t)\frac{\chi(x)+\chi\circ \tau(x)}{2}$$
$$=g(x)\int_{S}g(t)d\mu(t).$$ This completes the proof.
\end{proof}Now, we are ready to prove the secand main result of this
paper. We use the following  notations \cite{stkan}  :\\ -
$\mathcal{A}$ consists of the solution of $g:$ $S\longrightarrow
\mathbb{C}$ of d'Alembert's functional equation (\ref{eq6}) with
$\int_{S}g(t)d\mu(t)\neq 0$ and satisfying the conditions of Lemma
3.1(2)(ii).\\- To any $g\in \mathcal{A}$ we associate the function
$Tg=\int_{S}g(t)d\mu(t)g:$ $S\longrightarrow \mathbb{C}$.\\
- $\mathcal{K}$ consists of the non-zero solutions $f:$
$S\longrightarrow \mathbb{C}$ of integral Kannappan's functional
equation (\ref{eq010}).
\begin{thm}{(1)} $T$ is a bijection of
$\mathcal{A}$ onto $\mathcal{K}$. The inverse $T^{-1}$:
$\mathcal{K}\longrightarrow \mathcal{A}$ is defined by
$$(T^{-1}f)(x)=\frac{\int_{S}f(xt)d\mu(t)}{\int_{S}f(t)d\mu(t)}$$ for
all $f\in \mathcal{K}$ and $x\in S. $\\{(2)} Any non-zero solution
$f$: $S\longrightarrow \mathbb{C}$ of integral Kannappan's
functional equation (\ref{eq010}) is of the form
$f=\int_{S}g(t)d\mu(t)g$, where $g\in \mathcal{A}$. Furthermore,
$f(x)=\int_{S}g(xt)d\mu(t)=\int_{S}g(x\tau(t))d\mu(t)=\int_{S}g(t)d\mu(t)g(x)$
for all $x\in S.$\\(3) $f$ is abelian \cite{07}  if and only if $g$
is abelian.\\(4) If $S$ is equipped with a topology then $f$ is
continuous if and only if $g$ is continuous.
\end{thm} \begin{proof} If $g\in \mathcal{A}$, then
$$\int_{S}Tg(xyt)d\mu(t)+\int_{S}Tg(x\tau(y)t)d\mu(t)=\int_{S}g(s)d\mu(s)[\int_{S}g(xyt)d\mu(t)+\int_{S}g(x\tau(y)t)d\mu(t)]$$
$$=\int_{S}g(s)d\mu(s)[g(xy)\int_{S}g(t)d\mu(t)+g(x\tau(y))\int_{S}g(t)d\mu(t)]$$
$$=(\int_{S}g(s)d\mu(s))^{2}[2g(x)g(y)]=2Tg(x)Tg(y).$$ Furthermore, $\int_{S}Tg(s)d\mu(s)=(\int_{S}g(s)d\mu(s))^{2}\neq
0$. So, we get  $Tg\in \mathcal{K}$. \\
From [8, Lemma 3] the map $T$ is injective. Now, we will show that
$T$ is surjective. Let $f\in \mathcal{K}$ and define the function
$$g(x)=\frac{\int_{S}f(xt)d\mu(t)}{\int_{S}f(t)d\mu(t)}.$$ In view of (3.3) and (3.4) we have
$$(\int_{S}f(s)d\mu(s))^{2}[g(xy)+g(x\tau(y))]$$$$=\int_{S}f(s)d\mu(s)\int_{S}f(xyt)d\mu(t)+\int_{S}f(s)d\mu(s)\int_{S}f(x\tau(y)t)d\mu(t)$$
$$=\int_{S}\int_{S}\int_{S}f(xytsk)d\mu(t)d\mu(s)d\mu(k)+\int_{S}\int_{S}\int_{S}f(x\tau(y)t\tau(s)k)d\mu(t)d\mu(s)d\mu(k)$$
$$=\int_{S}\int_{S}[\int_{S}f(xtysk)d\mu(k)+\int_{S}f(xt\tau(ys)k)d\mu(k)]d\mu(t)d\mu(s)=2\int_{S}f(xt)d\mu(t)\int_{S}f(ys)d\mu(s)$$
$$=2(\int_{S}f(s)d\mu(s))^{2}g(x)g(y).$$
It follow that $g$ is a
solution of d'Alembert's functional equation (\ref{eq6}).\\
On the other hand, In view of (3.3) and (3.4) we have
$$(\int_{S}g(s)d\mu(s))^{2}=\int_{S}g(s)d\mu(s)\int_{S}g(t)d\mu(t)=\frac{1}{2}\int_{S}\int_{S}[g(st)+g(s\tau(t))]d\mu(s)d\mu(t)$$
$$=\frac{1}{2}\int_{S}\int_{S}[\frac{\int_{S}f(stk)d\mu(k)}{\int_{S}f(s)d\mu(s)}+\frac{\int_{S}f(s\tau(t)k)d\mu(k)}{\int_{S}f(s)d\mu(s)}]d\mu(s)d\mu(t)$$
$$=\frac{1}{2}[\frac{\int_{S}\int_{S}\int_{S}f(stk)d\mu(s)d\mu(t)d\mu(k)}{\int_{S}f(s)d\mu(s)}$$
$$+\frac{\int_{S}\int_{S}\int_{S}f(s\tau(t)k)d\mu(s)d\mu(t)d\mu(k)}{\int_{S}f(s)d\mu(s)}]$$
$$=\frac{1}{2}[\frac{\int_{S}f(s)d\mu(s)\int_{S}f(s)d\mu(s)}{\int_{S}f(s)d\mu(s)}$$
$$+\frac{\int_{S}f(s)d\mu(s)\int_{S}f(s)d\mu(s)}{\int_{S}f(s)d\mu(s)}]=\int_{S}f(s)d\mu(s),$$
and
$$\int_{S}g(st)d\mu(s)d\mu(t)=\frac{\int_{S}\int_{S}\int_{S}f(tsk)d\mu(t)d\mu(s)d\mu(k)}{\int_{S}f(s)d\mu(s)}$$
$$=\frac{\int_{S}f(t)d\mu(t)\int_{S}f(s)d\mu(s)}{\int_{S}f(s)d\mu(s)}=\int_{S}f(s)d\mu(s),$$ which proves that $g$ satisfies the conditions of Lemma 3.1(2)(ii).\\
Finally, $(\int_{S}g(s)d\mu(s))^{2}=\int_{S}f(s)d\mu(s)\neq 0$, then
$\int_{S}g(s)d\mu(s)\neq 0$. This completes the proof.
\end{proof}
\begin{cor}\cite{stkan} If $\mu=\delta_{z_0}$, where $z_0$ is a fixed element in the center of a semi group $S$. Then, any non-zero solution
$f$: $S\longrightarrow \mathbb{C}$ of  Kannappan's functional
equation (\ref{eq2}) is of the form $f=g(z_0)g$, where $g$ is a
 solution of d'Alembert's functional equation (\ref{eq6})
with $g(z_0)\neq 0$ and satisfying the conditions of Lemma 3.1((ii).
\end{cor}
\begin{prop} The non-zero abelian solutions of integral Kannappan's
functional equation (\ref{eq010}) are the functions of the form
$$f(x)=[\frac{\chi(x)+\chi(\tau(x))}{2}]\int_{S}\chi(t)d\mu(t),\; x\in
S,$$ where $\chi:$ $S\longrightarrow \mathbb{C}$ is a multiplicative
function such that\\ $\int_{S}\chi(t)d\mu(t)\neq 0$ and
$\int_{S}\chi(\tau(t))d\mu(t)=\int_{S}\chi(t)d\mu(t)$.
\end{prop}


\end{document}